\documentclass[a4paper,12pt,reqno]{amsart}

\usepackage{amsmath}
\usepackage{amscd}
\usepackage{amssymb}
\usepackage{mathrsfs}
\usepackage[left=2.5cm,right=2.5cm,bottom=3cm,top=3cm]{geometry}
\newtheorem{thm}{Theorem}[section]

\newtheorem{lem}[thm]{Lemma}
\newtheorem{prob}[thm]{Problem}
\newtheorem{prop}[thm]{Proposition}
\theoremstyle{definition}

\newtheorem{rem}[thm]{Remark}

\begin{document}

\title{Infinite-time singularity type of the K\"ahler-Ricci flow II}

\author{Yashan Zhang}

\address{Beijing International Center for Mathematical Research, Peking University, Beijing 100871, China}
\email{yashanzh@pku.edu.cn}
\thanks{The author is partially supported by China Postdoctoral Science Foundation grant 2018M641053}

\begin{abstract}
On a compact K\"ahler manifold with semi-ample canonical line bundle and Kodaira dimension one, we observe a relation between the infinite-time singularity type of the K\"ahler-Ricci flow and the characteristic indexes of singular fibers of the semi-ample fibration.
\end{abstract}

\maketitle

\section{Introduction}
We continue our study in \cite{Zys17} on the singularity type of long time solutions to the K\"ahler-Ricci flow on compact K\"ahler manifolds. Our main interest is the \emph{relation between infinite-time singularity type (more precisely, long time behavior of curvature) of the K\"ahler-Ricci flow and the geometric/complex structure of the underlying compact K\"ahler manifold}. 
\par Thanks to the maximal existence time theorem of the K\"ahler-Ricci flow \cite{C,TZo,Ts}, the existence of long time solutions is equivalent to nefness of the canonical line bundle. In this paper, as in \cite{Zys17}, motivated by the Abundance Conjecture (which predicts that if the canonical line bundle of an algebraic manifold is nef, then it is semi-ample), we will further restrict our discussions on $n$-dimensional K\"ahler manifold, denoted by $X$, with semi-ample canonical line bundle $K_X$, and so the Kodaira dimension $kod(X)\in\{0,1,\ldots,n\}$.
\par Given an arbitrary K\"ahler metric $\omega_0$ on $X$, let $\omega=\omega(t)_{t\in[0,\infty)}$ be the solution to the K\"ahler-Ricci flow
\begin{equation}\label{nkrf}
\partial_t\omega=-Ric(\omega)-\omega
\end{equation}
running from $\omega_0$. Recall from \cite{Ha93} that the infinite-time singularities of the K\"ahler-Ricci flow are divided into two types. Precisely, we say a long time solution of the K\"ahler-Ricci flow \eqref{nkrf} is of type \textrm{IIb} if
\begin{equation}\label{defn_ii}
\limsup_{t\to\infty}\left(\sup_{X}|Rm(\omega(t))|_{\omega(t)}\right)=\infty\nonumber
\end{equation}
and of type  \textrm{III} if
\begin{equation}\label{defn_iii}
\limsup_{t\to\infty}\left(\sup_{X}|Rm(\omega(t))|_{\omega(t)}\right)<\infty\nonumber.
\end{equation}
For $X$ with semi-ample canonical line bundle, we let
\begin{equation}\label{semiample}
f:X\to X_{can}\subset \mathbb{CP}^N
\end{equation}
be the semi-ample fibration with connected fibers induced by pluricanonical system of $K_X$, where $X_{can}$, a $kod(X)$-dimensional irreducible normal projective variety, is the canonical model of $X$, and $V\subset X_{can}$ be the singular set of $X_{can}$ together with the critical values of $f$, and $X_y=f^{-1}(y)$ be a smooth fiber for $y\in X_{can}\setminus V$. The singularity type of the K\"ahler-Ricci flow on such manifold $X$ has been classified except for the following case (see e.g. \cite[Section 6.3]{To.kawa} or \cite[Section 1]{Zys17} for an overview):\\

\begin{itemize}
\item[($\star$)]\emph{$0<kod(X)<n$ ($n\ge3$) and the generic fiber $X_y$ is a finite quotient of a torus and $V\neq\emptyset$}.\\
\end{itemize}

\par On the one hand, as expected by Tosatti and Y.G. Zhang in \cite[Section 1]{ToZyg} (also see \cite[Conjecture 6.7]{To.kawa} for a more general conjecture by Tosatti) and confirmed by our previous work \cite[Theorem 1.4]{Zys17}, the singularity type of the K\"ahler-Ricci flow in case ($\star$) does not depend on the choice of the initial K\"ahler metric and hence should only depend on the complex structure of $X$. This result in particular indicates that the following problem should be very natural: 
\begin{prob}\label{prob}
Classify/characterize the complex structure on $X$ in terms of infinite-time singularity types of the K\"ahler-Ricci flow.
\end{prob} 

\par In fact, in case ($\star$), since the curvature on any compact subset of $X\setminus V$ is uniformly bounded \cite{FZ,HeTo,ToZyg}, to determine the singularity type one only needs to analyze the long time behavior of curvature near singular set $f^{-1}(V)$ and hence the singularity type should be closely related to the properties of the singular fibers. For example, a result of Tosatti and Y.G. Zhang \cite[Proposition 1.4]{ToZyg} implies that if there exists a rational curve in some singular fiber of $f$ in case ($\star$), then the K\"ahler-Ricci flow on $X$ must develop type \textrm{II}b singularity, which in particular has been applied very successfully when $X$ in case ($\star$) is a minimal elliptic surface, see \cite[Theorem 1.6]{ToZyg}.\\

In this note we will focus on the $kod(X)=1$ case, in which the canonical model $X_{can}$ of $X$ is a closed smooth Riemann surface. Let $f:X\to X_{can}$ be the semi-ample fibration in \eqref{semiample}, $V=\{s_0,\ldots,s_L\}$ the finite set of isolated critical values of $f$ and $X_{s_l}$ the singular fiber over $s_l$. Therefore, in $kod(X)=1$ case one may more precisely formulate Problem \ref{prob} as: \emph{classify/characterize the singular fibers $X_{s_l}$ in terms of infinite-time singularity types of the K\"ahler-Ricci flow}.

\par Our main goal here is to find some restrictions on singular fibers when $X$ admits type \textrm{III} solution to the K\"ahler-Ricci flow.
To begin with, we first observe one more application of \cite[Proposition 1.4]{ToZyg} as follows.

\begin{prop}\label{prop}
Assume $X$ as in case ($\star$), Kodaira dimension $kod(X)$ is one and $X$ is projective. Let $V=\{s_0,\ldots,s_L\}\subset X_{can}$ is the set of critical values of $f$ in \eqref{semiample}. If the K\"ahler-Ricci flow on $X$ is of type \textrm{III}, then every singular fiber $X_{s_l}$ has only one irreducible component, say $F_l$, where $F_l$ is a normal projective variety of only canonical singularities, satisfies $K_{F_l}\sim_{\mathbb Q}0$ and contains no rational curve.
\end{prop}

To state the next restriction on singular fibers, recall that for every singular fiber $X_{s_l}$ we can naturally define (see Section \ref{pf} for precise definitions) the characteristic index of $X_{s_l}$ along $X$, which is a pair of two numbers, denoted as $(\beta_l,N_l)$, measuring the singularities of the singular fiber $X_{s_l}$. In general, $\beta_l\in\mathbb{Q}_{>0}$ and $N_l\in\{1,2,\ldots,n\}$. With these notations, the following is our main result, which provides an interesting relation between the singularity type of the K\"ahler-Ricci flow and the characteristic indexes of singular fibers.

\begin{thm}\label{mainthm_1}
Assume $X$ as in case ($\star$) and Kodaira dimension $kod(X)$ is one. Let $V=\{s_0,\ldots,s_L\}\subset X_{can}$ be the set of critical values of $f$ in \eqref{semiample}.  If the K\"ahler-Ricci flow on $X$ is of type \textrm{III} , then for every singular fiber $X_{s_l}$ the characteristic index $(\beta_{l},N_l)=(\frac{1}{m_l},1)$ for some positive integer $m_l$.
\end{thm}

Note that in Theorem \ref{mainthm_1} we do \emph{not} need to assume $X$ is projective. 
Our proof 
mainly makes use of a metric geometric approach, which in fact proves a slightly stronger result, see Theorem \ref{mainthm} below.

Proofs will be provided in Section \ref{pf}.

\section{Proofs}\label{pf}
To prove Theorem \ref{mainthm_1} let's first recall some definitions, following \cite[Definition 2.3.1]{CY} (also see \cite[Definition 6.2.1]{HN}). Assume $X$ is an $n$-dimensional ($n\ge2$) compact K\"ahler manifold with semi-ample canonical line bundle and Kodaira dimension one. Let $f:X\to X_{can}$ and $V=\{s_1,...,s_L\}$ as before. For any given $s_l\in V$, by applying a Hironaka's resolution of singularities, one obtains a birational morphism
$$\pi:\hat{X}\to X$$
such that $\hat X$ is smooth, $\pi:\hat X\setminus \pi^{-1}(X_{s_l})\to X\setminus X_{s_l}$ is biholomorphic and
$$\hat f:=f\circ\pi:\hat X\to X_{can}$$
is a holomorphic surjective map with connected fibers and the singular fiber
\begin{equation}
\hat X_{s_l}=\sum a_jE_j\nonumber,
\end{equation}
where $a_i\in\mathbb{N}_{\ge1}$, $E_i$'s are smooth irreducible divisors in $\hat X$ and have simple normal crossings. We also have the following ramification formula
$$K_{\hat X}=\pi^*K_X+\sum k_jE_j,$$
where $k_j\in\mathbb{N}_{\ge0}$ since $X$ is smooth. Then the \emph{log-canonical threshold} of $X_{s_l}$ along $X$ is defined to be $\beta_l=\min\{\frac{k_{j}+1}{a_j}\}$, the \emph{log-canonical multiplicity} of $X_{s_l}$ along $X$ is defined to be the maximal integer $N_l$ such that there are $N_l$ distinct $E_j$'s with $\frac{k_j+1}{a_j}=\beta_l$ and non-empty intersection, and the \emph{characteristic index} of $X_{s_l}$ along $X$ is defined to be the pair $(\beta_l,N_l)$. The characteristic index $(\beta_l,N_l)$ measures the singularities of the singular fiber $X_l$ along $X$. Note that the above definitions do not depend on the choice of resolution $\pi$, see \cite[Lemma 2.1, Definition 2.1]{Ccy} and \cite[Section 2.3]{CY}.

\par Then we will prove the following, which implies our main result Theorem \ref{mainthm_1} immediately.

\begin{thm}\label{mainthm}
Assume $X$ as in case ($\star$) and Kodaira dimension $kod(X)$ is one. Let $V=\{s_0,\ldots,s_L\}\subset X_{can}$ be the set of critical values of $f$ in \eqref{semiample}. If there exists a singular fiber $X_{s_{l_0}}$ such that  the characteristic index $(\beta_{l_0},N_{l_0})$ is not of the form $(\frac{1}{m},1)$ for some positive integer $m$, then for every neighborhood $U$ of $X_{s_{l_0}}$ and every solution $\omega(t)$ to the K\"ahler-Ricci flow on $X$ we have 
\begin{equation}\label{blowup}
\liminf_{t\to\infty}\left(\sup_Usec_{\omega(t)}\right)=\infty\nonumber.
\end{equation}
In particular, every solution is of type \textrm{II}b.
\end{thm}

To prove Theorem \ref{mainthm} we will make use of a metric geometric approach, in which a key point is that, thanks to \cite[Theorem 1]{Zys}, the informations about the characteristic indexes of singular fibers will be saved in Gromov-Hausdorff limits of the K\"ahler-Ricci flow solution, via Song-Tian's generalized K\"ahler-Einstein current.

\begin{proof}[Proof of Theorem \ref{mainthm}]
 Firstly, we just assume $X$ is a compact K\"ahler manifold with semi-ample canonical line bundle and Kodaira dimension one. Let $\omega(t)_{t\in[0,\infty)}$ be a solution to the K\"ahler-Ricci flow \eqref{nkrf} on $X$. Thanks to Song and Tian's fundamental works \cite{ST07,ST12} there exists a generalized K\"ahler-Einstein current $\omega_{GKE}$ on $X_{can}$ which is a closed positive $(1,1)$-current on $X_{can}$ and a smooth K\"ahler metric on $X_{can}\setminus V$ such that, as $t\to\infty$, $\omega(t)$ converges to $f^*\omega_{GKE}$ as currents on $X$; moreover, for any $K\subset\subset X\setminus f^{-1}(V)$ there holds that $\omega(t)\to f^*\omega_{GKE}$ in $C^0(K,\omega_0)$-topology \cite{FZ,HeTo,TWY}. 
\par On the other hand, the numbers $\beta_l$ and $N_l$ arise naturally in the metric behaviors of the limiting singular metric $\omega_{GKE}$ near its singular locus. Precisely, by \cite[Theorem 1]{Zys} (also see \cite[Lemma 3.4]{ST07} and \cite[Section 3.3]{He} for the special case that $X$ is a minimal elliptic surface), there exists a constant $C\ge1$ such that, for any $s_l\in V$, we can choose a local holomorphic chart $(\Delta,s)$ in $X_{can}$ near $s_l$ (in which we identify $\Delta$ with $\{s\in\mathbb{C}||s|\le1\}$ and $s_l$ with $0\in\mathbb{C}$) with the property that
\begin{equation}\label{est1.1}
C^{-1}|s|^{-2(1-\beta_l)}(-\log|s|)^{N_l-1}\le\frac{\omega_{GKE}}{\sqrt{-1}ds\wedge d\bar s}(s)\le C|s|^{-2(1-\beta_l)}(-\log|s|)^{N_l-1}
\end{equation}
holds on $\Delta\setminus\{0\}$. Here, $\beta_l$ and $N_l$ are the log-canonical threshold and the log-canonical multiplicity of the singular fiber $X_{s_l}$, respectively. Consequently, $(X_\infty,d_\infty)$, the metric completion of $(X_{can}\setminus V,\omega_{GKE})$, is a compact length metric space homeomorphic to $X_{can}$. 
\par Denote $(X_{can},d_{can}):=(X_\infty,d_\infty)$. 
\par From now on, we further assume the generic fiber of $f$ is a finite quotient of a torus. Thanks to a recent work of Tian and Z.L. Zhang \cite[Theorem 1.5, Corollary 1.6]{TZzl} (also see our previous work \cite[Theorem 2]{Zys} for a weaker version), for $X$ in case ($\star$) of Kodaira dimension one, $(X,\omega(t))\to(X_{can},d_{can})$ in Gromov-Hausdorff topology, as $t\to\infty$. 
\par Now assume Theorem \ref{mainthm} fails. We may assume without loss of generality that $s_{l_0}=s_0$. Then we have a solution $\omega(t)$ to the K\"ahler-Ricci flow \eqref{nkrf} on $X$, a small open neighborhood $U$ of $X_{s_0}$, a time sequence $t_i\to\infty$ and a constant $A'\ge1$ such that
\begin{equation}\label{A}
\sup_Usec_{\omega(t_i)}\le A'.
\end{equation}
We may assume $U\cap X_{s_l}=\emptyset$ for every $s_l\neq s_0$.

Since by a well-known result of Hamilton the scalar curvature of $\omega(t)$ is uniformly bounded from below, it follows from \eqref{A} that there exists a constant $A''\ge1$ with
\begin{equation}\label{A'}
\inf_Usec_{\omega(t_i)}\ge-A''
\end{equation}
for every $i$.

On the other hand, let's try to understand the limiting metric space $(X_{can},d_{can})$ from a metric geometric viewpoint. In fact, by combining \eqref{A} and \eqref{A'} we know $(X,\omega(t_i)$ has uniformly bounded sectional curvature on $U$; also, the diameter of $(X,\omega(t_i))$ is uniformly bounded from above. We fix an arbitrary point $x_0\in X_{s_0}$, then $(X,\omega(t_i),x_0)\to(X_{can},d_{can},s_0)$ in pointed Gromov-Hausdorff topology. We will need the following
\begin{lem}\label{lem}
There exist a sufficiently small constant $\rho>0$ and a sufficiently large constant $T>0$, such that for any $t_i\ge T$ there holds 
$$B_{\omega(t_i)}(x_0,\rho)\subset U.$$
\end{lem}
\begin{proof}[Proof of Lemma \ref{lem}]
This is essentially contained in \cite[Section 5, proof of Theorem 1.5]{TZzl} (also see \cite[Theorem 2, Section 4]{Zys}). In fact, we may firstly fix a small positive number $\delta$ such that $f^{-1}
(B_{d_{can}}(s_0,\delta))\subset U$, where $B_{d_{can}}(s_0,\delta)$ is the ball in $(X_{can},d_{can})$. We then set $\rho:=\frac{\delta}{4}$, By Gromov-Hausdorff convergence of the K\"ahler-Ricci flow one can fix a $T>0$ such that for any $t_i\ge T$, $f:(X,\omega(t_i))\to(X_{can},d_{can})$ is a $\frac{\delta}{4}$-Gromov-Hausdorff approximation. Now, for any $x\in B_{\omega(t_i)}(x_0,\rho)$, i.e. $d_{\omega(t_i)}(x,x_0)<\frac{\delta}{4}$, we have 
$$d_{can}(f(x),s_0)\le d_{\omega(t_i)}(x,x_0)+\frac{\delta}{4}\le \frac{\delta}{2},$$
and so $f(x)\in B_{d_{can}}(s_0,\delta)$, i.e. $x\in f^{-1}
(B_{d_{can}}(s_0,\delta))\subset U$. This lemma is proved.
\end{proof}
(We may also mention an observation which is not necessary in this note: for sufficiently $t_i$, $X_{s_0}\subset B_{\omega(t_i)}(x_0,\rho)$ since the diameter of fiber $X_{s_0}$ tends to zero as $t_i\to\infty$, see \cite[Section 5, proof of Theorem 1.5]{TZzl}.)
\par By Lemma \ref{lem}, we can fix a sufficiently large constant $A$ such that $(X,\omega(t_i),x_0)$ are $\{A\}$-regular at $x_0$ in the sense of Naber and Tian \cite[Definition 1.1]{NT}. Then, we can apply theory of Naber and Tian \cite[Theorem 1.1, Remark 1.2]{NT} to see that $(X_{can},d_{can})$ is a $C^{1}$ Riemannian orbifold in a neighborhood of $s_0$ (note that $X_{can}$ is of real $2$-dimension and hence there is no non-orbifold point in a neighborhood of $s_0$).  It is known that real 2-dimensional orbifold singularities are completely classified, see e.g. \cite[Sections 1 and 2]{Sc}. In our case, for the given $s_0$ in the singular locus $V$, it is an \emph{isolated} singular point; then by the forementioned classification results we know the associated orbifold group must be the finite cyclic group generated by a rotation, namely, $\mathbb Z_{m_0}$ for some integer $m_0\ge2$ depending on $s_0$ (see e.g. \cite[Sections 1 and 2]{Sc}). Consequently, locally near $s_0$, $(X_{can},d_{can})$ is isometric to the quotient of $(\Delta,g)$ by $\mathbb Z_{m_0}$, where $\Delta=\{u\in \mathbb C||u|<1\}$ and $g$ is some Riemannian metric on $\Delta$ which is $\mathbb Z_{m_0}$-invariant. 

\par  Let's first proceed by assuming an additional condition that the Riemannian metric $g$ is compatible with the complex structure of $(\Delta,u)$. Then we may write $g=h(u)du\otimes d\bar u$ for $u\in\Delta$. We want to rewrite $g$ with respect to the holomorphic coordinate $s$ in $\Delta$, which is the holomorphic coordinate used in \eqref{est1.1}. To this end, we first set a holomorphic coordinate $v=u^{m_0}$ in $\Delta$ and rewrite $g$ with respect to $v$. Then we easily conclude (see e.g. \cite[Section 6D]{NS}) that the metric $g$ in local holomorphic chart $(\Delta,v)$ reads $h(v^{\frac{1}{m_0}})m_0^{-2}\frac{dv\otimes d\bar v}{|v|^{2(1-\frac{1}{m_0})}}$ (note that $h(v^{\frac{1}{m_0}})$ is well-defined). After possibly shrinking $\Delta$, we may choose a biholomorphism $v=k(s):(\Delta,s)\to(\Delta,v)$; moreover, by theory on functions of one complex variable, there exists a holomorphic function $\tilde k=\tilde k(s)$ such that $|\tilde k(s)|$ is a positive function on $\Delta$ and $v=k(s)=\tilde k(s)s$. Therefore, the metric $g$ in local holomorphic chart $(\Delta,s)$ reads 
$$\frac{h(k(s)^{\frac{1}{m_0}})|\frac{\partial k}{\partial s}|^2}{m_0^2|\tilde k(s)|^{2(1-\frac{1}{m_0})}}\frac{ds\otimes d\bar s}{|s|^{2(1-\frac{1}{m_0})}},$$
and there exists a constant $\tilde C\ge1$ such that
\begin{equation}\label{est1.2}
\tilde C^{-1}\frac{ds\otimes d\bar s}{|s|^{2(1-\frac{1}{m_0})}}\le g\le\tilde C\frac{ds\otimes d\bar s}{|s|^{2(1-\frac{1}{m_0})}}
\end{equation}
on $\Delta\setminus\{0\}$. Here we have used that all the involved functions $h,|\frac{\partial k}{\partial s}|$ and $|\tilde k|$ are continuous, positive and bounded on $\Delta$.
\par Next we shall show that the inequality \eqref{est1.2} holds generally, even when $g$ is not necessarily compatible with the complex structure. In fact, we can always fix on $\Delta$ a $\mathbb Z_{m_0}$-invariant Riemannian metric $\tilde g$ which is quasi-isometric to $g$ (i.e. there is a constant $C\ge1$ such that $C^{-1}\tilde g\le g\le C\tilde g$ holds on $\Delta$) and compatible with the complex structure (for example, choose $\tilde g$ to be the standard Euclidean metric). Then the above arguments prove the inequality \eqref{est1.2} for $\tilde g$, and so for $g$ since $g$ is quasi-isometric to $\tilde g$. In conclusion, $g$ satisfies \eqref{est1.2} generally.

\par Finally, recall that $d_{can}$ on $X_{can}$ is induced by $\omega_{GKE}$ which satisfies \eqref{est1.1} in local holomorphic coordinate $(\Delta,s)$ near $s_0$. By comparing \eqref{est1.1} with \eqref{est1.2} there must hold $\beta_0=\frac{1}{m_0}$ and $N_0=1$. 

\par Theorem \ref{mainthm} is proved.
\end{proof}

\begin{rem}

(i) Obviously, Theorem \ref{mainthm_1} is an immediate consequence of Theorem \ref{mainthm}. Alternatively, we can prove Theorem \ref{mainthm_1} directly. Indeed, if we assume a type \textrm{III} solution $\omega(t)$ on $X$, then the curvature of $\omega(t)$ is uniformly bounded on $X$ and hence we can apply the same arguments as above, in which the role of \cite[Theorem 1.5]{TZzl} can be played by a weaker result \cite[Theorem 2]{Zys} since a type \textrm{III} solution automatically has a uniform lower bound for Ricci curvature on $X$.

(ii) We mention that while the proof of Gromov-Hausdorff convergence of the K\"ahler-Ricci flow uses only the upper bound of $\omega_{GKE}$ given in \eqref{est1.1} (see \cite[Theorem 1.5]{TZzl} or \cite[Theorem 2]{Zys}), the above proof of Theorem \ref{mainthm} really needs the \emph{asymptotics} of $\omega_{GKE}$ given in \eqref{est1.1}.

(iii)  From its proof, Theorems \ref{mainthm} and \ref{mainthm_1} also hold when $X$ is of dimension two, i.e. a minimal elliptic surface. 
Let $X$ be a minimal elliptic surface with elliptic fibration $f:X\to X_{can}$. Let $V=\{s_0,\ldots,s_L\}$ be the critical values set of $f$. According to Kodaira's table of singular fibers \cite{BHPV}, if $\beta_l=\frac{1}{m_l}$ for some positive integer $m_l$ and $N_l=1$, then the singular fiber $X_{s_l}$ must be one of Kodaira's types $mI_0, I_0^*,\textrm{II}^*,\textrm{III}^*$ and $\textrm{IV}^*$. Therefore, if there exists a singular fiber which is not of Kodaira's type $mI_0, I_0^*,\textrm{II}^*,\textrm{III}^*$ nor $\textrm{IV}^*$, then the K\"ahler-Ricci flow on $X$ must be of type \textrm{IIb}. This result partially recovers \cite[Theorem 1.6]{ToZyg}. 

(iv) The proof of Theorem \ref{mainthm} is essentially a contradiction argument. A natural question is can we have an effective argument for it? Precisely, given the condition in Theorem \ref{mainthm}, can we effectively estimate the blowup speed of curvature in terms of characteristic indexes?

(v) In the study of long time solutions to the K\"ahler-Ricci flow, two of the most important aspects are the convergence and curvature behavior at time-infinity. We know from previous works (e.g. \cite{FZ,G,TZzl,TWY,ToZyg,Zys}) that the curvature behavior at time-infinity is useful in obtaining/improving the convergence at time-infinity. Our above arguments conversely show that the convergence at time-infinity is also useful in understanding the curvature behavior at time-infinity. It should be very interesting to find more interplays between these two aspects of the K\"ahler-Ricci flow. 

\end{rem}

Finally, we give a

\begin{proof}[Proof of Proposition \ref{prop}]
As we recalled in Section 1, if there exists a type \textrm{III} solution to the K\"ahler-Ricci flow on $X$, then by \cite[Proposition 1.4]{ToZyg} every singular fiber $X_{s_l}$ must contain no rational curve, which in particular implies $X_{s_l}$ must contain no uniruled irreducible component. In this case, since $X$ is projective, according to Takayama's classification result, $X_{s_l}$ falls into the case (1.2) in \cite[Theorem 1.1]{Ta}, namely, $X_{s_l}=d_lF_l$, where $F_l$ is the only one irreducible component of $X_{s_l}$ ($d_l$ is the multiplicity of $F_l$), and $F_l$ is a normal projective variety of only canonical singularities, satisfies $K_{F_l}\sim_{\mathbb Q}0$ (here ``$\sim_{\mathbb Q}$" means $\mathbb Q$-linear equivalence) and contains no rational curve. Lastly, we remark that, when $X$ is projective, the smooth fiber of $f$ is a flat projective K\"ahler manifold \cite{ToZyg} and hence an Abelian variety up to a finite unramified covering.
\end{proof}

\section*{Acknowledgements}
The author thanks Professors Yin Jiang and Xiaochun Rong for useful conversations on convergence theory of Riemannian manifolds, Professors Frederick Fong, Shijin Zhang, Zhenlei Zhang and Zhou Zhang for their interest and comments on this note and the referee for helpful suggestions and comments.


\begin{thebibliography}{99}
\bibitem{BHPV} Barth, W., Hulek, K., Peters, C. and Van de Ven, A., Compact complex surfaces, 2nd edition, Springer, Berlin (2004)
\bibitem{C} Cao, H.-D., Deformation of K\"{a}hler metrics to K\"{a}hler-Einstein metrics on compact K\"{a}hler manifolds, Invent. Math. 81 (2), 359-372 (1985)
\bibitem{Ccy} Chi, C.-Y., Pseudonorms and theorems of Torelli type, J. Differential Geom. 104 (2016), 239-273
\bibitem{CY} Chi, C.-Y. and Yau, S.-T., A geometric approach to problems in birational geometry, Proc. Natl. Acad. Sci. USA 105 (2008), no. 48, 18696-18701
\bibitem{FZ} Fong, F.T.-H. and Zhang, Z., The collapsing rate of the K\"{a}hler-Ricci flow with regular infinite time singularities, J. Reine Angew. Math. 703 (2015), 95-113
\bibitem{G} Guo, B., \textit{On the K\"ahler Ricci flow on projective manifolds of general type}, Int. Math. Res. Not. 2017 (7), 2139-2171
\bibitem{HN} Halle, L.H. and Nicaise, J., Motivic zeta functions for degenerations of abelian varieties and Calabi-Yau varieties, Zeta functions in algebra and geometry, 233-259, Contemp. Math., 566, Amer. Math. Soc., Providence, RI, 2012
\bibitem{Ha93} Hamilton, R.S., The formation of singularities in the Ricci flow, Surveys in differential geometry, Vol. II (Combridge, MA, 1993), 7-136, Int. Press, Cambridge, MA, 1995
\bibitem{He} Hein, H.-J., Gravitational instantons from rational elliptic surfaces, J. Amer. Math. Soc. 25 (2012), no. 2, 355-393
\bibitem{HeTo} Hein, H.-J. and Tosatti, V., Remarks on the collapsing of torus fibered Calabi-Yau manifolds, Bull. Lond. Math. Soc. 47 (2015), no. 6, 1021-1027
\bibitem{NT} Naber, A. and Tian, G., Geometric structures of collapsing Riemannian manifolds I, Surveys in geometric analysis and relativity, 439-466, Adv. Lect. Math. (ALM), 20, Int. Press, Somerville, MA, 2011
\bibitem{NS} Nasatyr, B. and Steer, B., Orbifold Riemann surfaces and the Yang-Mills-Higgs equations, Ann. Scuola Norm. Sup. Pisa CI. Sci. (4) 22 (1995), no. 4, 595-643
\bibitem{Sc} Scott, P, The geometry of $3$-manifolds, Bull. London Math. Soc. 15 (1983), 401-487
\bibitem{ST07} Song, J. and Tian, G., The K\"{a}hler-Ricci flow on surfaces of positive Kodaira dimension, Invent. Math., 170, 609-653 (2006)
\bibitem{ST12} Song, J. and Tian, G., Canonical measures and the K\"{a}hler-Ricci flow, J. Amer. Math. Soc. 25 (2012), no. 2, 303-353
\bibitem{Ta} Takayama, S., On uniruled degenerations of algebraic varieties with trivial canonical divisor, Math. Z. 259 (2008), no. 3, 487-501
\bibitem{TZzl} Tian, G. and Zhang, Z.L., Relative volume comparison of Ricci flow and its applications, arXiv:1802.09506
\bibitem{TZo} Tian, G. and Zhang, Z., On the K\"ahler-Ricci flow on projective manifolds of general type, Chinese Ann. Math. Ser. B 27 (2006), no. 2, 179-192
\bibitem{To.kawa} Tosatti, V., KAWA lecture notes on the K\"ahler-Ricci flow, Ann. Fac. Sci. Toulouse Math. 27 (2018), no. 2, 285-376
\bibitem{TWY} Tosatti, V., Weinkove, B. and Yang, X., The K\"{a}hler-Ricci flow, Ricci-flat metrics and collapsing limits, Amer. J. Math. 140 (2018), no. 3, 653-698
\bibitem{ToZyg} Tosatti, V. and Zhang, Y.G., Infinite-time singularities of the K\"{a}hler-Ricci flow, Geom. Topol. 19 (2015), no. 5, 2925-2948
\bibitem{Ts} Tsuji, H., Existence and degeneration of K\"{a}hler-Einstein metrics on minimal algebraic varieties of general type, Math. Ann. 281, 123-133 (1988)
\bibitem{Zys17} Zhang, Y.S., Infinite-time singularity type of the K\"ahler-Ricci flow, J. Geom. Anal. (2017), https://doi.org/10.1007/s12220-017-9949-2
\bibitem{Zys} Zhang, Y.S., Collapsing limits of the K\"ahler-Ricci flow and the continuity method, Math. Ann. (2018), https://doi.org/10.1007/s00208-018-1676-x
\end{thebibliography}
\end{document}